\documentclass[reqno]{amsart}
\usepackage{amssymb}
\usepackage[colorlinks=true,urlcolor=blue,linkcolor=blue,citecolor=blue]{hyperref}
\usepackage{amsmath}
\newcommand{\eps}{\epsilon}

\newcommand{\F}{\mathbb{F}}
\newcommand{\T}{\mathcal{T}}
\newcommand{\Q}{\mathbb{Q}}
\newcommand{\Z}{\mathbb{Z}}
\newcommand{\Rs}{R^*}
\newcommand{\pl}{p\operatorname{-}\lim}
\newcommand{\ql}{q\operatorname{-}\lim}
\newcommand{\D}{{\mathcal G}}
\newcommand{\MI}{\mathcal{AMI}}
\newcommand{\N}{\mathbb{N}}
\newcommand{\AR}{{\mathcal A}_R}
\newcommand{\AZ}{{\mathcal A}_\mathbb{Z}}
\newcommand{\AN}{{\mathcal A}_\mathbb{N}}
\newcommand{\AK}{{\mathcal A}_K}

\newcommand{\FIID}{LID}
\newtheorem{theorem}{Theorem}[section]
\newtheorem{lemma}[theorem]{Lemma}
\newtheorem{corollary}[theorem]{Corollary}
\newtheorem{proposition}[theorem]{Proposition}
\theoremstyle{definition}
\newtheorem{definition}[theorem]{Definition}
\newtheorem{example}[theorem]{Example}
\theoremstyle{remark}
\newtheorem{conjecture}[theorem]{Conjecture}
\newtheorem{remark}[theorem]{Remark}

\numberwithin{equation}{section}

\begin{document}
\author{Vitaly Bergelson}
\author{Joel Moreira}

\title[Measure preserving actions of affine semigroups]{Measure preserving actions of affine semigroups and $\{x+y,xy\}$ patterns}
\begin{abstract}
Ergodic and combinatorial results obtained in \cite{Bergelson_Moreira15} involved measure preserving actions of the affine group $\AK$ of a countable field $K$.
In this paper we develop a new approach based on ultrafilter limits which allows one to refine and extend the results obtained in \cite{Bergelson_Moreira15} to a more general situation involving the measure preserving actions of the non-amenable affine semigroups of a large class of integral domains. (The results in \cite{Bergelson_Moreira15} heavily depend on the amenability of the affine group of a field).
Among other things, we obtain, as a corollary of an ultrafilter ergodic theorem, the following result:
Let $K$ be a number field and let ${\mathcal O}_K$ be the ring of integers of $K$. For any finite partition $K=C_1\cup\cdots\cup C_r$ there exists $i\in\{1,\dots,r\}$ and many $x\in K$ and $y\in{\mathcal O}_K$ such that $\{x+y,xy\}\subset C_i$.
\end{abstract}
\thanks{The first author gratefully acknowledges the support of the NSF under grants DMS-1162073 and DMS-1500575}

\maketitle
\section{Introduction}
One of the early results in Ramsey theory, due to I. Schur \cite{Schur16}, states that for any finite partition (or, as it is customary to say, coloring) $\N=C_1\cup\cdots\cup C_r$ of the natural numbers\footnote{In this paper we abide by the convention that $\N=\{1,2,3,\dots\}$.}, one of the cells $C_i$ contains a triple of the form $\{x,y,x+y\}$.
It is not hard to see that any finite coloring $\N=\bigcup_{i=1}^rC_i$ yields also a monochromatic triple of the form $\{x,y,xy\}$ (just observe that the restriction of a coloring of $\N$ to the set $\{2^n:n\in\N\}$ induces a new coloring of $\N$ and apply Schur's theorem).

A famous open conjecture states that for any finite coloring of $\N$, one finds (many) monochromatic quadruples of the form $\{x,y,x+y,xy\}$. Even a weaker version of this conjecture, asking for non-trivial monochromatic configurations of the form $\{x+y,xy\}$ is, so far, quite recalcitrant.
The above questions become more manageable if one considers finite partitions of the set of rational numbers $\Q$. An ergodic approach developed by the authors in \cite{Bergelson_Moreira15} shows that actually any `large' set in $\Q$ (and, indeed, in any countable field $K$) contains plenty of configurations of the form $\{x+y,xy\}$.

The results obtained in \cite{Bergelson_Moreira15} naturally lead to new questions which are addressed in this paper. In order to present the questions (and the answers) we need first to introduce pertinent notation and definitions and formulate some relevant results from \cite{Bergelson_Moreira15}.

Let $K$ be an infinite countable field.
For each $u\in K$ let $A_u:K\to K$ be the addition map $A_u:x\mapsto x+u$ and, for $u\neq0$, let $M_u:K\to K$ denote the multiplication map $M_u:x\mapsto ux$.
Let $\AK=\{A_uM_v:x\mapsto vx+u\mid u,v\in K,\;v\neq0\}$ denote the affine group of $K$.
A sequence $(F_N)_{N\in\N}$ of finite subsets of $K$ is a \emph{double F\o lner sequence} if it is asymptotically invariant under any fixed affine transformation $g\in\AK$.
Given any double F\o lner sequence $(F_N)_{n\in\N}$ in $K$ one can define the \emph{affinely invariant} upper density $\bar d_{(F_N)}(\cdot)$ by the formula
$$\bar d_{(F_N)}(E):=\limsup_{N\to\infty}\frac{|E\cap F_N|}{|F_N|},\ E\subset K$$
(The affine invariance means that $\bar d_{(F_N)}(E)=\bar d_{(F_N)}\big(f(E)\big)$ for any $f\in\AK$.)
The main ergodic theoretical result in \cite{Bergelson_Moreira15} is the following analogue of von Neumann's mean ergodic theorem:
\begin{theorem}\label{thm_introergodic}
Let $K$ be an infinite countable field, let $(U_g)_{g\in\AK}$ be a unitary representation of $\AK$ on a Hilbert space $\mathcal{H}$, let $I=\big\{f\in \mathcal{H}:(\forall g\in\AK)\ U_gf=f\big\}$ be the invariant subspace and let $P:\mathcal{H}\to I$ be the orthogonal projection onto $I$. Then for any $f\in \mathcal{H}$ and any double F\o lner sequence $(F_N)_{N\in\N}$ in $K$ we have
$$\lim_{N\to\infty}\frac1{|F_N|}\sum_{u\in F_N}U_{M_uA_{-u}}f=Pf$$
\end{theorem}

From Theorem \ref{thm_introergodic} we derived the following result:
\begin{theorem}\label{thm_introlargereturn}
Let $K$ be an infinite countable field, let $\big(X,{\mathcal B},\mu,(T_g)_{g\in\AK}\big)$ be a probability measure preserving system and let $B\in{\mathcal B}$.
Then, for any double F\o lner sequence $(F_N)_{N\in\N}$ in $K$ we have
\begin{equation}\label{eq_doublekhintchine}
  \lim_{N\to\infty}\frac1{|F_N|}\sum_{u\in F_N}\mu\big(T_{A_u}^{-1}B\cap T_{M_u}^{-1}B\big)\geq\mu(B)^2.
\end{equation}
\end{theorem}

\begin{corollary}\label{cor_introlargereturns}
  Let $K$ be an infinite countable field, let $\big(X,{\mathcal B},\mu,(T_g)_{g\in\AK}\big)$ be a probability measure preserving system and let $B\in{\mathcal B}$.
Then, for any $\delta\in(0,1)$, 
the set
\begin{equation}\label{eq_rbeps}R(B,\delta):=\Big\{u\in K:\mu\big(T_{M_u}^{-1}B\cap T_{A_u}^{-1}B\big)>\delta\mu(B)^2\Big\}\end{equation}
has positive upper density with respect to any double F\o lner sequence.
\end{corollary}

Using a version of Furstenberg's correspondence principle (see Theorem 2.8 in \cite{Bergelson_Moreira15}) we deduced from Theorem \ref{thm_introlargereturn} the following combinatorial corollary
\begin{corollary}\label{cor_intro4}
  Let $K$ be an infinite countable field, let $(F_N)_{N\in\N}$ be a double F\o lner sequence in $K$ and let $E\subset K$ be such that $\bar d_{(F_N)}(E)>0$. Then $E$ contains many pairs of the form $\{x+y,xy\}$.
\end{corollary}

Theorems \ref{thm_introergodic} and \ref{thm_introlargereturn} depend heavily on the amenability of the affine group $\AK$ and form a sort of the ultimate result that can be achieved via Ces\`aro averages.
Since the affine semigroups of rings (such as $\Z$ or, say, the polynomial ring $\F[t]$ where $\F$ is a finite field) are not amenable\footnote{See Proposition \ref{proposition_amenableringfield} below.}, it is a priori not clear what kind of statements similar to Theorems \ref{thm_introergodic} and \ref{thm_introlargereturn} and Corollaries \ref{cor_introlargereturns} and \ref{cor_intro4} can be formulated (and proved) if one replaces fields by more general rings.
In particular, one would like to know if the corresponding set $R(B,\delta)$ is "large" for any measure preserving action of the affine semigroup $\AZ$ of $\Z$.
As we will see below, an alternative approach, based on convergence along ultrafilters, not only allows one to have reasonable analogues of Theorems \ref{thm_introergodic} and \ref{thm_introlargereturn} for actions of $\AZ$, but also leads to a strong generalization of Corollary \ref{cor_introlargereturns} for actions of $\AK$ which guarantees the filter property of sets $R(B,\delta)$ (see Theorem \ref{thm_Rintersection} below for a precise formulation).

Observe that \eqref{eq_doublekhintchine} resembles a classical result of Khintchine (see, for example, \cite[Theorem 5.2]{Bergelson00}) stating that for any probability measure preserving system $(X,{\mathcal B},\mu,T)$ and any $B\in{\mathcal B}$
\begin{equation}\label{eq_introkhintchine}
  \lim_{N-M\to\infty}\frac1{N-M}\sum_{n=M}^N\mu(B\cap T^{-n}B)\geq\mu(B)^2
\end{equation}
Formula \eqref{eq_introkhintchine} in turn implies the so-called Khintchine's recurrence theorem, stating that the set
\begin{equation}\label{eq_sbdelta}
  S(B,\delta)=\{n:\mu(B\cap T^{-n}B)\geq\delta\mu(B)^2\}
\end{equation}
is syndetic for any $\delta\in(0,1)$. (A set $E\subset\Z$ is \emph{syndetic} if it has bounded gaps, in other words, if finitely many translates of $E$ cover $\Z$. More generally, a subset $E$ of a group is (left) \emph{syndetic} if finitely many translates of the form $gE$ cover $G$.)
Motivated by Khintchine's recurrence theorem, one would like to get a similar finite tiling property for sets of the form $R(B,\delta)$.

Corollary \ref{cor_introlargereturns} states that $R(B,\delta)$ has positive upper density with respect to any \emph{double} F\o lner sequence.
One can show (see Example \ref{example_thick} below) that sets which have positive density along any double F\o lner sequence are, in general, neither additively syndetic nor multiplicatively syndetic. Nevertheless, they still posses a strong enough tiling property which is revealed via the (a posteriori quite natural) notion of affine syndeticity:
\begin{definition}[Affine syndeticity]
  Given an infinite field $K$, a set $S\subset K$ is called \emph{affinely syndetic} if there exists a finite number of affine transformations $g_1,\dots,g_k\in\AK$ such that for any $x\in K$ at least one of the images $g_1(x),\dots,g_k(x)$ lies in $S$.
\end{definition}
The notion of affine syndeticity is explored in detail in Section \ref{sec_synd&thick}.
In particular we have the following proposition (cf. Theorem \ref{thm_syndensity} below).
\begin{proposition}
  Let $K$ be an infinite countable field.
  A subset $S\subset K$ is affinely syndetic if and only if it has positive upper density with respect to any double F\o lner sequence.
In particular, the sets $R(B,\delta)$ defined in \eqref{eq_rbeps} are affinely syndetic.
\end{proposition}

Observe that, in general, affinely syndetic sets do not have the finite intersection property. 
For example, the subsets of rational numbers defined by
$$E_1=\bigcup_{n\in\Z}[2n,2n+1)\subset\Q\qquad\qquad E_2=\bigcup_{n\in\Z}[2n-1,2n)\subset\Q$$
are both additively (hence affinely) syndetic, but have empty intersection.

On the other hand, one can show that the sets $S(B,\delta)$ appearing in \eqref{eq_sbdelta} do have the finite intersection property, 
although the easiest way of proving this involves either the so-called IP-limits or limits along idempotent ultrafilters (we note that these `non-Ces\`arian' limits work well also when one deals with large returns along polynomials, see \cite{Bergelson_Furstenberg_McCutcheon96}, \cite[Section 3]{Bergelson96} and \cite{Bergelson_Robertson14}).

The above discussion suggests that the sets $R(B,\delta)$ may have the finite intersection property as well.
The following theorem provides a confirmation of this feeling.
The class of \FIID{} rings (where \FIID{} stands for Large Ideal Domain) which appears in its formulation is defined in the beginning of the next section, and includes $\Z$ and the polynomial ring $\F[x]$ over any finite field $\F$ as rather special cases.
\begin{theorem}\label{thm_Rintersection}
Let $R$ be a \FIID, let $t\in\N$ and, for each $i=1,...,t$, let $(\Omega_i,\mu_i)$ be a probability space, let $(T_g^{(i)})_{g\in\AR}$ be a measure preserving action of the affine semigroup $\AR$ of $R$ on $(\Omega_i,\mu_i)$ and let $B_i\subset\Omega_i$ be a measurable set with positive measure.
Let $\delta\in(0,1)$ and let $R(B_i,\delta)$ be defined as in equation (\ref{eq_rbeps}) with respect to the action $(T_g^{(i)})_{g\in\AR}$.
Then the intersection
\begin{equation}\label{eq_thm_R_intersection}R(B_1,\delta)\cap...\cap R(B_t,\delta)\end{equation}
is affinely syndetic (and, in particular, nonempty).
\end{theorem}
Theorem \ref{thm_Rintersection} is proved in Section \ref{sec_proofmain}, where it is obtained as a corollary of an ultrafilter analogue of Corollary \ref{cor_introlargereturns} (see Theorem \ref{thm_mcentral}).
Roughly speaking, Theorem \ref{thm_mcentral} asserts that given an ultrafilter $p$ with certain rich combinatorial properties and an isometric anti-representation $(U_g)_{g\in\AR}$ of the affine semigroup $\AR$ on a Hilbert space ${\mathcal H}$, we have $\pl_u U_{M_uA_{-u}}f=Vf$, where\footnote{The symbol $\pl$ denotes limit along ultrafilter $p$. See Section \ref{sec_ultrafilters} for the relevant background on ultrafilters.} $V:{\mathcal H}\to{\mathcal H}$ is an orthogonal projection.
This in turn allows us to obtain, as a corollary, the following analogue of formulas \eqref{eq_doublekhintchine} and \eqref{eq_introkhintchine} for measure preserving actions $(T_g)_{g\in\AR}$ of $\AR$:
$$\pl_u\mu(T_{A_u}^{-1}B\cap T_{M_u}^{-1}B)\geq\mu(B)^2$$

\begin{remark}
  To appreciate the power of the ultrafilter approach, one should note that the Ces\`aro convergence results established in \cite{Bergelson_Moreira15} imply only the affine syndeticity of the intersections
\begin{equation}\label{eq_rmrk}
R\big(B_1,0\big)\cap...\cap R\big(B_t,0\big)
\end{equation}
of return sets $R(B_i,0)$, rather than the affine syndeticity of the intersection of the `optimal' return sets $R(B_i,\delta)$, as in \eqref{eq_thm_R_intersection}.
\end{remark}

Juxtaposing the (still unsolved) problem of finding monochromatic $\{x+y,xy\}$ patterns in $\N$ with the positive result contained in Corollary \ref{cor_intro4}, we see that there is a place for an `intermediate' result which would guarantee, for any finite coloring of $\Q$, the existence of a monochromatic configuration of the form $\{x+n,xn\}$ where $x\in\Q$, $n\in\N$.
As we will see, results of this kind can be obtained via ultrafilter methods developed in this paper. 
In particular, we have the following special cases of a more general Theorem \ref{cor_integer}, to be found in Section \ref{sec_proofmain}:

\begin{theorem}\label{cor_introinteger}
\

  \begin{enumerate}
    \item For any finite partition $\Q=C_1\cup\cdots\cup C_r$ of the rational numbers, there exists a cell $i\in\{1,\dots,r\}$ and many $x\in Q$, $n\in\N$ such that $\{x+n,xn\}\subset C_i$.
    \item More generally, if $K$ is a number field and $\mathcal{O}_K$ is its ring of integers, for any finite partition $K=C_1\cup\cdots\cup C_r$, there exists a cell $i\in\{1,\dots,r\}$ and many $x\in K$, $n\in\mathcal{O}_K$ such that $\{x+n,xn\}\subset C_i$.
    \item Let $\F$ be a finite field, let $K$ denote the field of rational functions (i.e. quotients of polynomials) over $\F$ and let $\F[x]$ denote the ring of polynomials.
        Then for any finite partition $K=C_1\cup\cdots\cup C_r$, there exists a cell $i\in\{1,\dots,r\}$ and many $f\in K$, $g\in\F[x]$ such that $\{f+g,fg\}\subset C_i$.
  \end{enumerate}
\end{theorem}

The paper is organized as follows.
In Section \ref{sec_preli_affine} we define the class of \FIID{} rings and present some general facts about affine semigroups. In particular, we prove that the affine semigroup of a countable integral domain $R$ is amenable if and only if $R$ is a field. 
In Section \ref{sec_ultrafilters} we provide the necessary background on ultrafilters, and introduce the notion of $DC$ sets, which will play a fundamental role in the rest of the paper.
In Section \ref{sec_synd&thick} we introduce the notions of affinely thick and affinely syndetic, explore some of the properties of these families of sets and connect these notions with $DC$ sets.
In Section \ref{sec_proofmain} we state and prove the main theorems.
Finally, in Section \ref{sec_misc} we discuss some notions of largeness pertinent to the study of $\{x+y,xy\}$ patterns and formulate a conjecture which, if true, implies that for any finite partition of $\N$, one of the cells of the partition contains plenty of configurations $\{x+y,xy\}$.

\section{Preliminaries: large ideal domains, affine semigroups, double F\o lner sequences}\label{sec_preli_affine}
Throughout this paper we will work with a special class of rings:
\begin{definition}\label{def_fiid}
  A ring $R$ is called a \emph{large ideal domain} (\FIID) if it is an infinite countable integral domain and for any $x\in R\setminus\{0\}$, the ideal $xR$ is a finite index additive subgroup of $R$.
\end{definition}
Every field is trivially an \FIID.
The following proposition gives some non-trivial examples of \FIID{} rings.
\begin{proposition}\label{prop_fiidexample}The following rings are \FIID:
\begin{enumerate}
  \item Any integral domain $R$ whose underlying additive group is finitely generated.
  In particular, the ring of integers $\mathcal{O}_K$ of a number field $K$ satisfies this property.
  \item The ring of polynomials $\F[x]$ over a finite field $\F$.
\end{enumerate}
  \end{proposition}

\begin{proof}
  \begin{enumerate}
    \item 
    Since $(R,+)$ is an infinite finitely generated abelian group, it contains torsion-free elements and therefore the identity $1_R$ of $R$ has infinite order in $(R,+)$.
    If some element $x\in R$ had torsion, say $nx=0$ for some $n\in\N$, then $(n1_R)x=0$, contradicting the absence of $0$ divisors.
    Using the classification of finitely generated abelian groups we can now represent $(R,+)$ as $\Z^d$ for some $d\in\N$.
    
    For any non-zero $x\in R$, the map $\phi:y\mapsto xy$ is an injective endomorphism of $(R,+)$ (injectivity follows from the absence of divisors of $0$) whose image $\phi(R)$ is the ideal $xR$.
    We claim that the image of any injective homomorphism $\phi:\Z^d\to\Z^d$ has a finite index in $\Z^d$, which will finish the proof.

    Indeed, representing $\phi$ as a matrix, injectivity implies that the determinant of $\phi$ is non-zero.
    Therefore it has an inverse $\phi^{-1}$ with entries in $\Q$.
    Multiplying $\phi^{-1}$ by the least common multiple $n$ of its entries we obtain a matrix $n\phi^{-1}$ with coefficients in $\Z$.
    Therefore $n\Z^d=(n\phi^{-1})\phi(\Z^d)\subset \phi(\Z^d)$, so $[\Z^d:\phi(\Z^d)]\leq[\Z^d:n\Z^d]=n^d<\infty$, proving the claim.

    \item Let $f\in\F[x]$ have degree $d$.
    For any $g\in\F[x]$ one can divide $g$ by $f$ and obtain $g=fq+r$ where $\deg r<d$.
    Therefore $g-r$ belongs to the ideal $f\F[x]$.
    It follows that the set of polynomials $r$ with degree smaller than $d$ form a complete set of coset representatives for $f\F[x]$.
    Since $\F$ is finite, there are only finitely many such representatives and hence the index of $f\F[x]$ is finite as desired.
  \end{enumerate}
\end{proof}
\begin{remark}
  There are number fields whose ring of integers is not a principal ideal domain (PID).
Hence, part (1) of Proposition \ref{prop_fiidexample} includes some \FIID{} which are not PID.
We also observe that not every PID is a \FIID.
Indeed, the ring $\Q[x]$ of all polynomials with rational coefficients is a PID, but the ideal $x\Q[x]$ has infinite index as an additive subgroup of $\Q[x]$, so $\Q[x]$ is not a \FIID.
\end{remark}

Some of the results in this paper are true only for fields; we will indicate the distinction in each case and we will use the letter $K$ to denote a field.

Let $R$ be a ring, we denote by $\Rs$ the set of its non-zero elements.
An \emph{affine transformation} of $R$ is a map $f:R\to R$ of the form $f(x)=ux+v$ with $u\in\Rs,v\in R$.
The \emph{affine semigroup} of $R$ is the semigroup of all affine transformations of $R$ (the semigroup operation being composition of functions) and will be denoted by $\AR$.
Observe that $\AR$ is a group if and only if $R$ is a field.

For each $v\in R$, the map $x\mapsto x+v$ will be denoted by $A_v$ (add $v$) and, for each $u\in\Rs$, the map $x\mapsto ux$ will be denoted by $M_u$  (multiply by $u$).
Note that the distributive law in $R$ can be expressed as:
\begin{equation}
  \label{eq_distr}
  M_uA_v=A_{uv}M_u
\end{equation}
The affine transformations $A_v$ with $v\in R$ form the \emph{additive subgroup} of $\AR$, denoted by $S_A$.
The affine transformations $M_u$ with $u\in\Rs$ form the \emph{multiplicative sub-semigroup} of $\AR$, denoted by $S_M$.
Observe that $S_A$ is isomorphic to the additive group $(R,+)$ and $S_M$ is isomorphic to the multiplicative semigroup $(\Rs,\cdot)$.

Note that the map $x\mapsto ux+v$ is the composition $A_vM_u$.
Thus the sub-semigroups $S_M$ and $S_A$ generate the semigroup $\AR$.
When $K$ is a field, $\AK$ is the semidirect product of the (abelian) groups $S_A$ and $S_M$ and hence is amenable.
However, as it was pointed out in Remark 6.2 in \cite{Bergelson_Moreira15}, the semigroup $\AZ$ is not amenable.
In fact we have:

\begin{proposition}\label{proposition_amenableringfield}
  Let $R$ be a countable integral domain.
  The affine semigroup $\AR$ is amenable if and only if $R$ is a field.
\end{proposition}

\begin{proof}
  As was explained above, if $R$ is a field then $\AR$ is amenable.
  Assume now that $\AR$ is amenable.
  The semigroup $\AR$ acts naturally on $R$ by affine transformations, therefore the amenability of $\AR$ implies the existence of a finitely additive mean $\lambda:{\mathcal P}(R)\to[0,1]$
defined on all the subsets of $R$ which is invariant under all affine transformations (this means that $\lambda\big(\{x\in R:g(x)\in E\}\big)=\lambda(E)$ for any $E\subset R$ and $g\in\AR$).
  Given $x\in R^*$, we have $1=\lambda(R)=\lambda(xR)$ (because the map $y\mapsto xy$ belongs to $\AR$).

  Assume, for the sake of a contradiction, that $R$ is not a field and let $x\in R^*$ be a non-invertible element.
  The ideal $xR$ is not the whole ring and hence there is a shift $xR+a$ which is disjoint from $xR$.
  The invariance of $\lambda$ implies that $\lambda(xR)=\lambda(xR+a)$, but disjointness implies that $\lambda\big(xR\cup(xR+a)\big)=\lambda(xR)+\lambda(xR+a)=2\lambda(xR)$.
  We now conclude that
  $$1=\lambda(xR)=\frac12\lambda\big(xR\cup(xR+a)\big)\leq\frac12\lambda(R)=\frac12$$
  which gives the desired contradiction.
\end{proof}
When $g\in\AR$ is an affine transformation of $R$ and $E\subset R$ is any subset, we define
\begin{equation}\label{eq_deftau}\theta_gE=\left\{g(x):x\in E\right\}\qquad\text{and}\qquad\theta_g^{-1}E=\{x\in R:g(x)\in E\}\end{equation}
Throughout this paper, in order to make the notation less cumbersome, and when no confusion can arise, we will adopt the following convention: Let $(T_g)_{g\in\AR}$ be a measure preserving action of $\AR$ (on some probability space) and let $(U_g)_{g\in\AR}$ be a isometric (anti-)representation of $\AR$ (on some Hilbert space).
For $v\in R$ and $u\in\Rs$ we will write $A_v$ instead of $\theta_{A_v}$, $T_{A_v}$ or $U_{A_v}$, and $M_u$ instead of $\theta_{M_u}$, $T_{M_u}$ or $U_{M_u}$.

\begin{definition}
Let $K$ be a field.
  A \emph{double F\o lner sequence} in $K$ is a sequence $(F_N)$ of finite subsets of $K$ such that for every $u\in K^*$ we have
  $$\lim_{N\to\infty}\frac{|F_N\cap(F_N+u)|}{|F_N|}=\lim_{N\to\infty}\frac{|F_N\cap(F_Nu)|}{|F_N|}=1$$
\end{definition}
It follows from \cite[Proposition 2.4]{Bergelson_Moreira15} that double F\o lner sequences exist in any countable field $K$.
This fact also follows from Theorem \ref{thm_syndensity} below.

\begin{definition}
  Let $K$ be a field, let $E\subset K$ and let $(F_N)$ be a double F\o lner sequence in $K$.
  The \emph{upper density} of $E$ with respect to $(F_N)$ is
  $$\bar d_{(F_N)}(E):=\limsup_{N\to\infty}\frac{|E\cap F_N|}{|F_N|}$$
  and the \emph{lower density} of $E$ with respect to $(F_N)$ is
  $$\underline d_{(F_N)}(E):=\liminf_{N\to\infty}\frac{|E\cap F_N|}{|F_N|}$$

\end{definition}
Several basic properties of the upper and lower densities with respect to a F\o lner sequence in a group remain true for densities with respect to double F\o lner sequences, and the proofs carry over to this setting.
We list some of these facts in the next lemma.

 \begin{lemma}\label{lemma_density}
 Let $K$ be a field, let $(F_N)$ be a double F\o lner sequence in $K$, let $E_1,E_2\subset K$ and let $g\in\AK$.
 \begin{enumerate}
 \item\label{lemma_density_shift} $\bar d_{(F_N)}(\theta_gE)=\bar d_{(F_N)}(E)$ and $\underline d_{(F_N)}(\theta_gE)=\underline d_{(F_N)}(E)$.
 \item\label{lemma_density_upperunion} $\bar d_{(F_N)}(E_1\cup E_2)\leq\bar d_{(F_N)}(E_1)+\bar d_{(F_N)}(E_2)$
 \item\label{lemma_density_lowerunion} $\underline d_{(F_N)}(E_1\cup E_2)\geq\underline d_{(F_N)}(E_1)+\underline d_{(F_N)}(E_2)$.
 \item\label{lemma_density_sum} If $E_2=K\setminus E_1$, then $\bar d_{(F_N)}(E_1)+\underline d_{(F_N)}(E_2)=1$.
 \end{enumerate}
 \end{lemma}

\section{Auxiliary results involving ultrafilters}\label{sec_ultrafilters}To prove Theorems \ref{thm_Rintersection} and \ref{cor_introinteger} we will use ultrafilters on $R$.
For the reader's convenience, we provide in this section a brief review of necessary ultrafilter background.
For a more detailed account see \cite{Bergelson10} and, for a comprehensive treatment, see \cite{Hindman_Strauss98}.
\begin{definition}Let $X$ be a countable infinite set.
An ultrafilter on $X$ is a family $p$ of subsets of $X$ such that
\begin{itemize}
\item $X\in p$.
\item If $E_1\in p$ and $E_1\subset E_2$ then $E_2\in p$.
\item If $E_1\in p$ and $E_2\in p$ then $E_1\cap E_2\in p$.
\item $E\in p\iff (R\setminus E)\notin p$.
\end{itemize}
The set of all ultrafilters on $X$ is denoted by $\beta X$.
\end{definition}
For any $u\in X$, the \emph{principal ultrafilter} $p_u$ is  defined by the rule $E\in p_u\iff u\in E$.
By a slight abuse of notation we will often denote $p_u$ by $u$.

The set $\beta X$ of all ultrafilters on $X$ can be identified with the Stone-\v{C}ech compactification of the (discrete) set $X$ (see Theorem 3.27 in \cite{Hindman_Strauss98}).
The space $\beta X$ is a compact Hausdorff space (cf. Theorem 3.18 in \cite{Hindman_Strauss98}) with the topology generated by the clopen sets
\begin{equation}\label{eq_closure}\overline{E}:=\{p\in\beta X:E\in p\}\qquad\qquad\forall E\subset X\end{equation}
Let $p\in\beta X$ be an ultrafilter, let $Y$ be a compact Hausdorff space and let $f:X\to Y$ be a function.
It is not hard to check that there exists a unique point $y\in Y$ such that for every neighborhood $U$ of $y$ we have $\{u\in X:f(u)\in U\}\in p$.
We denote this by $\displaystyle\pl_uf(u)=y$ (one can also write $\displaystyle y=\lim_{u\to p}f(u)$, but we stick with the former notation since it is more suggestive of the analogy with Ces\`aro limits).

Now let $X=R$ be a ring.
One can extend the operations of addition and multiplication from $R$ to $\beta R$ as follows.
Given $p,q\in\beta R$ we define
\begin{equation}\label{eq_ultrasum}p+q=\{E\subset R:\{u\in R:A_u^{-1}E\in q\}\in p\}
\end{equation}
\begin{equation}\label{eq_ultraproduct}pq=\{E\subset R:\{u\in R:M_u^{-1}E\in q\}\in p\}
\end{equation}
The operations defined by (\ref{eq_ultrasum}) and (\ref{eq_ultraproduct}) are associative in $\beta R$ (cf. Theorems 4.1, 4.4 and 4.12 in \cite{Hindman_Strauss98}).
However (for the rings we deal with) these operations do not commute and fail to satisfy the distributive law.
Nevertheless, we have
\begin{proposition}\label{prop_ultradistr}
  Let $u\in R$ and $p,q\in\beta R$. Then
  \begin{itemize}
    \item $u+p=p+u$ and $up=pu$.
    \item $(p+q)u=pu+qu$.
  \end{itemize}
\end{proposition}

 One can easily check that for each $p,q\in\beta R$ we have (cf. Remark 4.2 in \cite{Hindman_Strauss98}):
\begin{equation}\label{eq_ultralimit}p+q=\pl_u(u+q)\qquad\qquad pq=\pl_u(uq)\end{equation}

An ultrafilter $p\in\beta R$ is an \emph{additive idempotent} if $p+p=p$, and it is a \emph{multiplicative idempotent} if $pp=p$.
Observe that $1\in\beta R$ is a multiplicative idempotent and $0\in\beta R$ is both an additive idempotent and a multiplicative idempotent.
The following fundamental result due to Ellis (see, for instance, Theorem 3.3 in \cite{Bergelson96}) guarantees the existence of idempotents in any compact semigroup.
\begin{lemma}\label{lemma_ellis}
Let $(S,\circ)$ be a compact Hausdorff semigroup such that for each $s\in S$ the function $x\mapsto x\circ s$ from $S$ to itself is continuous. Then there exists $s\in S$ such that $s\circ s=s$.
\end{lemma}
In what follows, Lemma \ref{lemma_ellis} will be repeatedly applied to closed sub-semigroups of $(\beta R,+)$ and $(\beta(\Rs),\cdot)$.

Since $R$ is an integral domain and $\beta(\Rs)=(\beta R)\setminus\{0\}$ is closed in $\beta R$, it follows from \eqref{eq_ultralimit} that $\beta(\Rs)$ is closed under multiplication.
In view of Proposition \ref{prop_ultradistr} and \eqref{eq_ultralimit} we have that, for each $u\in R$, both maps $A_u:p\mapsto p+u$ and $M_u:p\mapsto pu$ are continuous.
Therefore we can define topological dynamical systems $(\beta R,S_A)$ and $(\beta(\Rs),S_M)$, where $S_A$ and $S_M$ are the additive and multiplicative sub-semigroups of $\AR$, respectively (cf. Section \ref{sec_preli_affine}).
Invoking again \eqref{eq_ultralimit} one can check that any closed $S_A$-invariant subset of $\beta R$ is a semigroup for addition, and any closed $S_M$-invariant subset of $\beta\Rs$ is a semigroup for multiplication.

By Zorn's lemma, there exist minimal non-empty compact $S_A$-invariant subsets of $\beta R$ and minimal non-empty compact $S_M$-invariant subsets of $\beta(\Rs)$.
An \emph{additive minimal idempotent} is a non-principal ultrafilter $p\in\beta R$ which belongs to a minimal compact $S_A$-invariant set and such that $p+p=p$.
A \emph{multiplicative minimal idempotent} is a non-principal ultrafilter $p\in\beta(\Rs)$ which belongs to a minimal compact $S_M$-invariant set and such that $pp=p$.

\begin{definition}\label{def_ami}
  Let $R$ be a ring. We denote by $\MI$ the set of all additive minimal idempotents in $\beta R$ and we denote by $\mathcal{MMI}$ the set of all multiplicative minimal idempotents in $\beta(\Rs)$.
\end{definition}
A set $C\subset R$ is called \emph{additively central} if there exists $p\in\MI$ such that $C\in p$.
Similarly, any member of an ultrafilter $p\in\mathcal{MMI}$ is called \emph{multiplicatively central}\footnote{The notion of \emph{central set} in $\Z$ was introduced by Furstenberg in topologico-dynamical terms \cite{Furstenberg81}. 
Furstenberg's definition of central sets makes sense in any semigroup (see \cite[Definition 6.2]{Bergelson_Hindman90}).
 One can show (see \cite[Theorems 6.8 and 6.11]{Bergelson_Hindman90}) that a subset of a countable semigroup is central if and only if it belongs to a minimal idempotent ultrafilter.}.
In this paper we are interested in sets $C\subset R$ which are simultaneously additively and multiplicatively central.

Unfortunately, the sets $\MI$ and $\mathcal{MMI}$ are in general disjoint (cf \cite[Corollary 13.15]{Hindman_Strauss98}).
However, at least when $R$ is an \FIID, the closure $\overline{\MI}$ has non-trivial intersection with $\mathcal{MMI}$ (see Proposition \ref{prop_thickcentral} below).

\begin{definition}
\label{def_central}
\
\begin{itemize}
\item Let $\D=\overline{\MI}\cap\mathcal{MMI}$.
\item A set $C\subset R$ is called $DC$ (double central) if there exists an ultrafilter $p\in\D$ such that $C\in p$.
\item A set $C\subset R$ is called $DC^*$ if it has non-empty intersection with every $DC$ set\footnote{We call the reader's attention to the fact that there is no relation between the $^*$ in $DC^*$ and the $^*$ in $R^*$.}.
\end{itemize}
\end{definition}

Observe that a set $C\subset R$ is $DC^*$ if and only if it is contained in every ultrafilter $p\in\D$ (this follows directly from Definition \ref{def_central} and the definition of ultrafilters).

We will need four more facts about ultrafilters which do not appear in the literature in the form that we need.
Lemma \ref{lemma_centralIP0} is the adaptation of Theorem 3.5 from \cite{Bergelson_Hindman94}, where the analogous result is proved for $\N$. The proof carries over to our setup.

\begin{lemma}\label{lemma_centralIP0}
Let $R$ be a countable integral domain, let $p\in\mathcal{MMI}$ and let $B\in p$.
Then for every $r\in\N$ there exists a set $Z\subset R$ with cardinality $|Z|=r$ and such that the set of finite sums of $Z$ satisfies
$$FS(Z):=\left\{\sum_{i\in Z'}i\ \Big|\ \emptyset\neq Z'\subset Z\right\}\subset B$$
\end{lemma}
\begin{proof}

Let $T\subset\beta R$ be the collection of all non-principal ultrafilters $p$ such that any member $A\in p$ contains a set of the form $FS(Z)$ with $Z$ having arbitrarily large cardinality (sets $A$ satisfying this property are called IP$_0$ sets).
It follows from Theorem 5.8 in \cite{Hindman_Strauss98} that every additive idempotent is in $T$, so $T$ is non-empty.

Since $p\in\mathcal{MMI}$, there exists some minimal subsystem $(Y,S_M)$ of $(\beta\Rs,S_M)$ such that $p\in Y$.
We claim that $Y\cap T$ is non-empty.

Let $q\in T$.
We have that $up=M_up\in Y$ for every $u\in\Rs$. It follows from equation (\ref{eq_ultralimit}) and the fact that $Y$ is closed that $qp\in Y$ as well.
Let $E\in qp$.
By definition, $\{u\in R: M_u^{-1}E\in p\}\in q$.
Thus for each $r\in\N$ there exists $Z\subset K$ with $|Z|=r$ and such that $FS(Z)\subset\{u\in R:M_u^{-1}E\in p\}$.
Since $FS(Z)$ is finite, the intersection $\bigcap_{u\in FS(Z)}M_u^{-1}E$ is also in $p$ and hence is infinite.
Let $a$ be a non-zero element in that intersection; we have that $a\in M_u^{-1}E$ for every $u\in FS(Z)$ and hence $FS(Z)a=FS(Za)\subset E$.
Observe that $|Za|=|Z|$ because there are no divisors of $0$.
Since $E\subset qp$ and $|Z|$ were chosen arbitrarily, we conclude that $qp\in T$.
This proves the claim.

Next, let $q\in Y\cap T$ and let $u\in\Rs$.
We trivially have $uq\in Y$.
Furthermore, if $A\in uq$ then $M_u^{-1}A\in q$ and hence if contains $FS(Z)$ for a set $Z$ of arbitrary finite cardinality.
But then $A$ contains $M_uFS(Z)=FS(uZ)$ and hence $uq\in T$.
This implies that $uq\in Y\cap T$ and hence $(Y\cap T,S_M)$ is a subsystem of $(\beta\Rs,S_M)$.
Since $(Y,S_M)$ is a minimal system, we conclude that $Y\cap T=Y$.
This implies that $Y\subset T$.
Hence $p\in T$ and we are done.
\end{proof}

We will also need the following technical lemma
\begin{lemma}\label{lemma_finiteindexip*}
Let $G$ be a group and let $H\subset G$ be a normal subgroup with finite index.
Then for any ultrafilter $p\in\beta G$ in the closure of the idempotents we have $H\in p$.
\end{lemma}
\begin{proof}
  The set of ultrafilters containing $H$ is a closed set, hence we can assume that $p$ is itself an idempotent.
  Since $H$ has only finitely many cosets, exactly one of them, say $aH$ is in $p$.
  Therefore, given $g\in G$ we have $g^{-1}aH\in p$ if and only if $g^{-1}a\in aH$.
  This is equivalent to $g\in aHa^{-1}=H$ (because $H$ is normal).
  Since $aH\in p=p+p$ we conclude
  $$\{g\in G:g^{-1}aH\in p\}\in p\iff H\in p$$
\end{proof}

A particular case of Lemma \ref{lemma_finiteindexip*} is when $R$ is a \FIID{}, $H$ is a non-trivial ideal and $p\in\D$.
If $p\in\beta(\Rs)$ contains an ideal $bR$ for some $b\in\Rs$, then one can define an ultrafilter $b^{-1}p$ as the family of sets $E\subset R$ such that $bE\subset p$. Observe that in this case $bq=p$.

The following lemma is the analogue of Theorem 5.4 in \cite{Bergelson_Hindman90} (where it is stated and proved for $\N$).

\begin{lemma}\label{lemma_upinmi}
Let $R$ be a \FIID, let $p\in\overline{\MI}$ and let $u\in\Rs$.
Then both $up$ and $u^{-1}p$ belong to $\overline{\MI}$.
\end{lemma}

\begin{proof}

Since $M_u:p\mapsto up$ and $M_u^{-1}:p\mapsto u^{-1}p$ are continuous (on their respective domains), it suffices to show that if $p\in\MI$ then also both $up$ and $u^{-1}p$ are in $\MI$.
It follows directly from Proposition \ref{prop_ultradistr} that $up+up=u(p+p)=up$, so $up$ is an additive idempotent.
Checking the definitions easily yields that $u^{-1}p$ is an additive idempotent.

All that remains to show is that $up$ and $u^{-1}p$ belong to minimal subsystems of $(\beta R,S_A)$.
\begin{itemize}

  \item[(1)]$u^{-1}p\in\MI$

 Let $X=\overline{\{v+p:v\in R\}}$ be the minimal compact $S_A$-invariant subset of $\beta R$ such that $p\in X$.
 It is not hard to check that the set $u^{-1}X:=\{q\in\beta R:bq\in X\}$ is $S_A$-invariant, compact, and contains $u^{-1}p$.

Since $R$ is a \FIID, the ideal $uR$ has finite index as an additive subgroup of $R$.
Therefore there exists a finite set $F\subset R$ of coset representatives such that $R=F+uR$.
Choose $F$ minimal with this property and such that $F\cap uR=\{0\}$.

If $Z\subset u^{-1}X$ is any compact $S_A$-invariant subset, than $F+uZ$ is a compact subset of $X$.
 We now show that $F+uZ$ is also invariant.
 Indeed, observe that any $v\in R$ can be decomposed as $v=a+uv'$ with $a\in F$ and $v'\in R$; thus if $a_1+uz\in F+uZ$ is arbitrary (with $a_1\in F$ and $z\in Z$) and $v_1\in R$, then $v_1+a_1+uz=a+uv'+uz=a+u(v'+z)\in F+uZ$ by invariance of $Z$.

Since $X$ is minimal, this implies that $F+uZ$ is either empty (in which case $Z$ is empty) or coincides with $X$.
In the second case we claim that $Z=u^{-1}X$.
Indeed, let $q\in u^{-1}X$, then it satisfies $uq\in X=F+uZ$, whence $uq=a+uz$ for some $a\in F$ and $z\in Z$.
Therefore $uR$ is in both $uz$ and $a+uz$ which implies that $a\in uR\cap uR=\{0\}$.
This means that $uq\in uZ$ and hence $q\in Z$, proving the claim.

It follows that $u^{-1}X$ is a compact minimal $S_A$-invariant subset of $\beta R$.
 Since $u^{-1}p\in u^{-1}X$ it follows that $u^{-1}p\in\MI$ as desired.

  \item[(2)] $up\in\MI$
\

Let $Y=\overline{\{v+up:v\in R\}}\subset\beta R$.
It suffices to show that $Y$ is itself minimal (compact and $S_A$-invariant being immediate consequences of its construction).
Recalling that $F\subset R$ is a finite set such that $R=F+uR$, we can rewrite
$$Y=\overline{\big\{(a+uv)+up:a\in F;v\in R\big\}}=F+u\overline{\{v+p:v\in R\}}=F+uX,$$
where in the second equality we used Proposition \ref{prop_ultradistr}.
Let $Z\subset Y$ be a non-empty compact $S_A$-invariant subset; we need to show that $Z=Y$.
Let $Z_1=\{q\in X:uq\in Z\}=X\cap u^{-1}Z$.

We claim that $F+uZ_1=Z$.
It is clear that $F+uZ_1\subset Z$ (for $Z$ is $S_A$-invariant).
Next let $q\in Z$ be arbitrary, we need to show that $q\in F+uZ_1$.
There is exactly one $a\in F$ such that $a+uR\in q$.
Let $r$ be the ultrafilter defined by $E\in r\iff a+uE\in q$ (observe that $r$ is indeed an ultrafilter because $a+uR\in q$ and hence $R\in r$), we will show that $r\in X$.
Indeed let $E\in r$, since $a+uE\in q\in Y$, we have that $v+a+uE\in up$ for some $v\in R$.
By definition this means that $u^{-1}(v+a+uE)\in p$, so $v+a\in uR$ and $u^{-1}(v+a)+E\in p$.
Finally this implies that $E\in -u^{-1}(v+a)+p$, and since $E\in r$ was arbitrary it follows that $r\in\overline{\{v'+p:v'\in R\}}=X$ as desired.
Next observe that $a+ur=q\in Z$.
Since $Z$ is invariant, this implies that $ur\in Z$ as well, and hence $r\in Z_1$, so $q=a+ur\in F+uZ_1$ as desired.

Since $Z$ is non-empty, it follows that $Z_1$ is non-empty.
Next we show that $Z_1$ is $S_A$-invariant.
For any $v\in R$ and $q\in Z_1$ we have $u(v+q)=uv+uq\in uv+Z\subset Z$ since $Z$ is invariant, so $v+q\in Z_1$ as desired.
Since $Z_1\subset X$ and $X$ is minimal we have $Z_1=X$.
But this means that $Z=F+uZ_1=F+uX=Y$ and hence $Y$ is minimal as desired.
\end{itemize}
\end{proof}

\begin{lemma}\label{lemma_kplim}
Let $X$ be a compact space and let $(x_u)_{u\in R}$ be a sequence in $X$ indexed by a countable ring $R$.
Then for each $k\in\Rs$ and $p\in\beta R$ we have $\displaystyle\pl_ux_{ku}=k\pl_ux_u$.
\end{lemma}
\begin{proof}
Let $\displaystyle x=\pl_ux_{ku}$ and let $U\subset X$ be a neighborhood of $x$.
By definition, the set $E=\{u\in R:x_{ku}\in U\}\in p$.
Note that $E=\{u\in R:x_u\in U\}/k$, and hence $\{u\in R:x_u\in U\}\in kp$.
Since $U$ is an arbitrary neighborhood of $x$ we conclude that $\displaystyle k\pl_ux_u=x$.
\end{proof}

\section{Affine syndeticity and thickness}\label{sec_synd&thick}
In this section we will develop the notions of affinely syndetic and affinely thick subsets of $R$.
The definitions and proofs are parallel to the usual notions of syndetic and thick.
Recall that, for a discrete semigroup $G$, a set $S\subset G$ is \emph{syndetic} if finitely many translates of $S$ cover $G$.
More precisely, $S$ is (left) syndetic in $G$ if there exists a finite set $F\subset G$ such that every $g\in G$ can be written as $g=xs$ with $s\in S$ and $x\in F$.

Recall from equation (\ref{eq_deftau}) the notation $\theta_gE=\{g(x):x\in E\}$ for a set $E\subset R$ and $g\in\AR$.
When $F\subset\AR$, $S\subset R$ and $x\in R$ we write
$$\theta_F^{-1}S:=\bigcup_{g\in F}\theta_{g}^{-1}S\qquad\text{ and }\qquad\theta_Fx:=\bigcup_{g\in F}g(x)$$
We slightly generalize here the definition of affine syndeticity, given in the Introduction for fields, to general rings:
\begin{definition}
Let $R$ be a ring.
A set $S\subset R$ is \emph{affinely syndetic} if there exists a finite set $F\subset\AR$ such that $\theta_F^{-1}S=R$.
\end{definition}

Observe that if a set $S\subset\Rs$ is syndetic in either the group $(R,+)$ or the semigroup $(\Rs,\cdot)$, then $S$ is affinely syndetic.
Indeed, assume, for instance, that $S$ is syndetic in $(R,+)$ and let $F\subset R$ be a finite set such that $S-F=R$.
Then considering the subset $\{A_u:u\in F\}\subset\AR$ we deduce that $\theta_F^{-1}S=R$ and hence $S$ is affinely syndetic.
On the other hand, $S$ can be affinely syndetic and not be syndetic for neither the group $(R,+)$ nor the semigroup $(\Rs,\cdot)$ (this follows from Example \ref{example_thick} and Proposition \ref{prop_syndthick} below).

Recall that, for a discrete semigroup $G$, a set $T\subset G$ is \emph{thick} if it contains a shift of an arbitrary finite set.
More precisely, $T$ is (right) thick in $G$ if for every finite set $F\subset G$ there exists $g\in G$ such that $Fg\subset T$.
\begin{definition}
A set $T\subset R$ is \emph{affinely thick} if for every finite set $F\subset\AR$ there exists $x\in R$ such that $\theta_Fx\subset T$.
\end{definition}

Observe that if $T\subset R$ is affinely thick, then it is thick in both the group $(R,+)$ and the semigroup $(\Rs,\cdot)$.
The following example shows that there exist sets $T$ which are not affinely thick (even when $R$ is a field) but thick in both $(R,+)$ and $(\Rs,\cdot)$:
\begin{example}\label{example_thick}We take the ring $R=\Q$ of rational numbers.
Let $(G_N)$ be an increasing sequence of finite subsets of $\Q$ whose union is $\Q$.
For any sequence $(a_N)\subset\Q^*$, the set
$$E=\left(\bigcup_{N=1}^\infty\big(a_{2N-1}+G_{2N-1}\big)\right)\cup \left(\bigcup_{N=1}^\infty\big(a_{2N}G_{2N}\big)\right)= \bigcup_{N=1}^\infty E_N$$ is additively thick and multiplicatively thick, where $E_N=a_N+G_N$ when $N$ is odd and $E_N=a_NG_N$ when $N$ is even.
However, if $(a_N)$ is growing sufficiently fast, then $E$ is not affinely thick.
Indeed, for every point $x\in\Q$ we may have
$$\theta_{\{Id,A_1M_2\}}x=\{x,x+1,2x\}\not\subset E$$

To see this, let $a_0=1$ and $E_0:=\{0\}$.
Let $\Delta G_N$ denote the set defined by $\Delta G_N=\{x_2-x_1,x_3-x_2,\dots,x_k-x_{k-1}\}$ where $x_1<x_2<\dots<x_k$ is an ordering of the elements of $G_N$.
Let $M_N=\min\big\{|x|:x\in G_N\setminus\{0\}\big\}$.
Define recursively
$$a_N=\left\{\begin{array}{ccc}2\max\left(E_{N-1}\right)+\max\left(G_N\right)-2\min\left(G_N\right)&\text{if}&N\text{ is odd}\\\frac1{\min(\Delta G_N)}+\frac{2\max(E_{N-1})}{M_N}&\text{if}&N\text{ is even}\end{array}\right.$$

Note that if $N$ is even and $x\in E_N$, then $x+1\notin E_N$.
If $N$ is odd and $x\in E_N$, then $x\geq\min(G_N)+a_N$ which implies that $2x>\max(G_N)+a_N$ and hence $2x\notin E_N$.
Thus for any $N\in\N$ and $x\in\Q$, the set $\{x,x+1,2x\}$ is not a subset of $E_N$.

Since $\min\{|x|:x\in E_{N+1}\setminus\{0\}\}>2\max\{|x|:x\in E_N\}$, if $x\in E_N$, then $2x\notin E_{N+1}$ (and in fact $2x\notin E_L$ for any $L>N$) and hence $\{x,x+1,2x\}$ is not a subset of $E$ for any $x\in\Q$
\end{example}

The following proposition is an immediate consequence of the definitions.
\begin{proposition}\label{prop_syndthick}
  A set $S\subset R$ is affinely syndetic if and only if it has non-empty intersection with every affinely thick set.
  A set $T\subset R$ is affinely thick if and only if it has non-empty intersection with every affinely syndetic set.
\end{proposition}

Now we connect affine syndeticity and thickness in countable fields with upper and lower density with respect to double F\o lner sequences.
\begin{theorem}\label{thm_syndensity}Let $K$ be a countable field.
  A set $S\subset K$ is affinely syndetic if and only if for every double F\o lner sequence $(F_N)$ in $K$, we have
  $\bar d_{(F_N)}(S)>0$.
  A set $S\subset K$ is affinely thick if and only if there exists a double F\o lner sequence $(F_N)$ in $K$ such that
  $\underline d_{(F_N)}(T)=1$.
\end{theorem}
\begin{proof}
  Assume $S\subset K$ is affinely syndetic and let $F\subset\AK$ be a finite set such that $\theta_F^{-1}S=K$.
  Then for any double F\o lner sequence $(F_N)$, using parts (\ref{lemma_density_shift}) and (\ref{lemma_density_upperunion}) of Lemma \ref{lemma_density} we have
  $$1=\bar d_{(F_N)}(K)=\bar d_{(F_N)}\left(\bigcup_{g\in F}\theta_{g^{-1}}S\right)\leq\sum_{g\in F}\bar d_{(F_N)}(\theta_{g^{-1}}S)=|F|\bar d_{(F_N)}(S)$$
  and hence $\bar d_{(F_N)}(S)\geq1/|F|>0$.

  Now assume that $T\subset K$ is affinely thick and let $(G_N)$ be an arbitrary (left) F\o lner sequence in $\AK$.
  For each $N\in\N$ let $x_N\in K$ be such that $F_N:=\theta_{G_N}x_N\subset T$ and $|F_N|=|G_N|$.
  To see why this is possible, note that for any affine transformations $g_1,g_2\in\AK$ with $g_1\neq g_2$, there is at most one solution $x\in K$ to the equation $g_1(x)=g_2(x)$.
  Thus there are only finitely many $x\in K$ such that $g_1x=g_2x$ for some pair $g_1\neq g_2\in G_N$.
  On the other hand, since $T$ is affinely thick, there are infinitely many $x\in K$ such that $\theta_{G_N}x\subset T$ (and indeed an affinely thick set of such $x$).

  We now show that $(F_N)$ is a double F\o lner sequence in $K$.
  For any fixed $g\in\AK$ we have
  $$F_N\cap \theta_gF_N=\theta_{G_N}x_N\cap \theta_g(\theta_{G_N}x_N)\supset \theta_{G_N\cap gG_N}x_N$$
   and hence
  $$1\geq\limsup_{N\to\infty}\frac{|F_N\cap gF_N|}{|F_N|}\geq\liminf_{N\to\infty}\frac{|F_N\cap gF_N|}{|F_N|}\geq\lim_{N\to\infty}\frac{|G_N\cap gG_N|}{|G_N|}=1$$
  because $(G_N)$ is a left F\o lner sequence in $\AK$.
  This implies that $(F_N)$ is a double F\o lner sequence in $K$.
  Since for each $N\in\N$ we have $F_N\subset T$ we conclude that $\underline d_{(F_N)}(T)=1$.

  Now if $S$ is not syndetic then it follows from Proposition \ref{prop_syndthick} that $K\setminus S$ is thick. Therefore there exits a double F\o lner sequence $(F_N)$ such that $\underline d_{(F_N)}(K\setminus S)=1$. From part (\ref{lemma_density_sum}) of Lemma \ref{lemma_density} if follows that $\bar d_{(F_N)}(S)=0$.

  Finally, if $T$ is not thick, then $K\setminus T$ is syndetic and hence for every double F\o lner sequence $(F_N)$ we have $\bar d_{(F_N)}(K\setminus T)>0$. By part (\ref{lemma_density_sum}) of Lemma \ref{lemma_density} we have $\underline d_{(F_N)}(T)<1$ for every double F\o lner sequence in $K$.

\end{proof}
\begin{remark}\label{rmrk_syndetic}
In view of Theorem \ref{thm_syndensity}, it follows from (the proof of) Theorem 2.5 in \cite{Bergelson_Moreira15} that the sets of return times $R(B,\eps)$ defined in (\ref{eq_rbeps}) are affinely syndetic.
The main idea behind the proof of Theorem \ref{thm_Rintersection} is that the sets $R(B,\eps)$ are not only affinely syndetic, but actually $DC^*$.
\end{remark}

In every countable semigroup, any thick set is central.
The same phenomenon occurs in our situation:

\begin{proposition}\label{prop_thickcentral}
Assume $R$ be a \FIID.
Then every affinely thick set in $R$ is $DC$ (see Definition \ref{def_central}).
\end{proposition}
\begin{proof}
Let $T\subset R$ be an affinely thick set.
For $g\in\AR$ define $\overline{\theta_{g^{-1}}T}\subset\beta R$ by equations (\ref{eq_deftau}) and (\ref{eq_closure}).
Note that, for any finite set $F\subset\AR$:
$$\bigcap_{g\in F}\overline{\theta_{g^{-1}}T}=\overline{\bigcap_{g\in F}\theta_{g^{-1}}T}=\overline{\{x\in R:\theta_Fx\subset T\}}$$
Since $T$ is affinely thick, the family of compact sets $\left\{\overline{\theta_{g^{-1}}T}:g\in\AR\right\}$ has the finite intersection property, and hence the intersection $\T:=\bigcap_{g\in\AR}\overline{\theta_{g^{-1}}T}$ is a non-empty compact subset of $\beta R$. We have the following description of $\T$:
$$p\in\T\iff(\forall g\in\AR)p\in\overline{\theta_{g^{-1}}T}\iff(\forall g\in\AR)\theta_{g^{-1}}T\in p$$
If $p,q\in\T$, we claim that both $p+q\in\T$ and $pq\in\T$.
Indeed, for all $g\in\AR$ and $u\in R$ we have $A_u^{-1}\theta_{g^{-1}}T=(\theta_gA_u)^{-1}T$. Therefore we have:
$$\theta_{g^{-1}}T\in p+q\iff\{u\in R:A_u^{-1}\theta_{g^{-1}}T\in q\}\in p\iff \{u\in R:(\theta_gA_u)^{-1}T\in q\}\in p$$
Since $q\in\T$ the set $\{u\in R:(\theta_gA_u)^{-1}T\in q\}=R\in p$, so we conclude that $p+q\in\T$.
The same argument with obvious modifications implies that $pq\in\T$ proving the claim.

We now have that $(\T,S_A)$ is a topological dynamical system.
Hence by Zorn's lemma there exists a minimal subsystem.
It follows from (\ref{eq_ultralimit}) that each minimal subsystem is actually an (additive) left ideal in $\beta R$, and hence, in view of Lemma \ref{lemma_ellis}, there exist (additive) minimal idempotents in $\T$.
Therefore the intersection $\T_1:=\overline{\MI}\cap\T$ is a non-empty compact subset of $\T$.

If $u\in\Rs$ and $p\in\T_1$, it follows from Lemma \ref{lemma_upinmi} that $up\in\overline{\MI}$, and thus $up\in\T_1$.
This means that $(\T_1,S_M)$ is a topological dynamical system and hence by Zorn's lemma it has minimal subsystems.
By Ellis theorem each minimal system (=left ideal) contains some multiplicative idempotent.
Let $p$ be a multiplicative minimal idempotent in $\T_1$.
Since $\T_1\subset\T$ we conclude that $T\in p$.
Since $\T_1\subset\overline{\MI}$ we conclude that $p\in\overline{\MI}$, and hence $p\in\D$.
\end{proof}

\begin{remark}\label{rmrk_DC*syndetic}
An immediate consequence of Propositions \ref{prop_thickcentral} and \ref{prop_syndthick} is that every $DC^*$ set is affinely syndetic.
\end{remark}

\section{Finite intersection property of sets of return times}\label{sec_proofmain}

In this section we study isometric anti-representations\footnote{We deal here with anti-representations instead of (a priori more natural) representations because a measure preserving action $(T_g)_{g\in G}$ of a non-commutative semigroup $G$ induces a natural anti-representation of $G$ by isometries on the corresponding $L^2$ space. Of course, the results obtained in this section hold true for isometric representations as well.} $(U_g)_{g\in\AR}$ of the affine semigroup $\AR$ of a ring $R$ on a Hilbert space $\mathcal{H}$ (this means that $\langle U_g\phi,U_g\psi\rangle=\langle\phi,\psi\rangle$ and $U_g(U_h\phi)=U_{hg}\phi$ for any $g,h\in\AR$ and $\phi,\psi\in \mathcal{H}$).

Recall that if $G$ is a semigroup and $(U_g)_{g\in G}$ is an isometric (anti-)representation of $G$ on a Hilbert space $\mathcal{H}$, then a vector $\phi\in \mathcal{H}$ is called \emph{compact} if the orbit $\{U_g\phi:g\in G\}\subset \mathcal{H}$ is pre-compact in the norm topology.
It is easy to see that the set of compact vectors is a closed subspace.

When $G$ is the additive sub-semigroup $S_A$ of the affine semigroup $\AR$, we denote the orthogonal projection onto the space of compact vectors by $V_A$ and when $G$ is the multiplicative sub-semigroup $S_M$ of the affine semigroup $\AR$, we denote the orthogonal projection onto the space of compact vectors by $V_M$.
Our main ergodic-theoretic result is the following analogue of Theorem \ref{thm_introergodic}, with Ces\`aro averages (which are unavailable in our current situation) replaced with limits along ultrafilters $p\in\D=\overline{\MI}\cap\mathcal{MMI}$.
\begin{theorem}\label{thm_main}Let $R$ be an \FIID{} (see Definition \ref{def_fiid}), let $\mathcal{H}$ be a Hilbert space and let $(U_g)_{g\in\AR}$ be an isometric anti-representation of $\AR$ on $\mathcal{H}$.
Then, for any $\phi,\psi\in \mathcal{H}$ and $p\in\D$ (see Definition \ref{def_central}) we have
$$\pl_u \langle A_u\phi,M_u\psi\rangle=\langle V_A\phi,V_M\psi\rangle.$$
\end{theorem}

In this section we will always work under the assumptions of Theorem \ref{thm_main}.
\subsection{Projection onto the space of compact vectors}
We have the following result:
\begin{lemma}\label{lemma_kronecker}
If $p\in\D$ (see Definition \ref{def_ami}) and $\phi\in \mathcal{H}$ then
$$V_M\phi=\pl_uM_u\phi\qquad\text{in the topology of weak convergence}$$
If $p\in\overline{\MI}$ and $k\in\Rs$ then
$$V_A\phi=\pl_uA_{ku}\phi\qquad\text{in the topology of weak convergence.}$$
\end{lemma}
\begin{proof}
Since $p\in\mathcal{MMI}$, the first equality follows\footnote{In \cite{Bergelson03} the results are stated and proved for groups only, but it is easy to check that the proofs work for discrete semigroups as well (as is observed in the first paragraph after the remark following Theorem 4.1 in \cite{Bergelson03}).} from Corollary 4.6 on \cite{Bergelson03}.
By the same corollary we have that $\displaystyle V_A\phi=\ql_uA_u\phi$ for every additive minimal idempotent $q$.

It follows from Lemma \ref{lemma_kplim} that $\displaystyle\pl_uA_{ku}\phi=k\pl_uA_u\phi$.
In view of Lemma \ref{lemma_upinmi} we have that $kp\in\overline{\MI}$. Since the map $\displaystyle q\mapsto\ql_uA_u\phi$ is continuous we conclude that
$$\pl_uA_{ku}\phi=k\pl_uA_u\phi=V_A\phi$$
\end{proof}

\begin{lemma}\label{lemma_Vcomute}
For every $\phi\in \mathcal{H}$ we have $V_AV_M\phi=V_MV_A\phi$.
\end{lemma}

\begin{proof}
Let $p\in\D$.
For each $k\in\Rs$, it follows from Lemma \ref{lemma_kronecker} that
$$M_kV_A\phi=M_k\left(\pl_uA_u\phi\right)=\pl_uM_kA_u\phi=\pl_uA_{ku}M_k\phi=V_AM_k\phi$$
Therefore
$$V_MV_Af=\pl_kM_kV_A\phi=\pl_kV_AM_k\phi=V_A\left(\pl_kM_k\phi\right)=V_AV_M\phi$$
\end{proof}

In view of Lemma \ref{lemma_Vcomute}, the operator $V:=V_AV_M$ is an orthogonal projection.
This gives the following simple corollary of Lemma \ref{lemma_Vcomute} which will be needed in the proof of Theorem \ref{thm_mcentral} below.
\begin{corollary}\label{lemma_positive}
Let $\phi,\psi\in \mathcal{H}$ and assume that $U_g\psi=\psi$ for every $g\in\AR$.
Then
$$\|\psi\|^2\cdot\langle V_A\phi,V_M\phi\rangle\geq\left|\langle\phi,\psi\rangle\right|^2$$
\end{corollary}

\begin{proof}
We have
\begin{eqnarray*}
\|\psi\|^2\cdot\langle V_A\phi,V_M\phi\rangle&=&\|\psi\|^2\cdot\langle V\phi,\phi\rangle= \|\psi\|^2\cdot\|V\phi\|^2\\&\geq&\big|\langle V\phi,\psi\rangle\big|^2=\big|\langle\phi,V\psi\rangle\big|^2= \big|\langle\phi,\psi\rangle\big|^2
\end{eqnarray*}
where the inequality follows from Cauchy-Schwarz inequality.
\end{proof}

\subsection{Dealing with $V_A\phi$}
The scheme of the proof of Theorem \ref{thm_main} is as follows: first we decompose $\phi=V_A\phi+\phi^\perp$ into its `additively compact' and `additively weak mixing' components.
Observe that $V_A(V_A\phi)=V_A\phi$ and $V_A\phi^\perp=0$.
The two main steps are to show that $\pl_u\langle A_uV_A\phi,M_u\psi\rangle=\langle V_A\phi,V_M\psi\rangle$ and that $\pl_u\langle A_uV_A\phi^\perp,M_u\psi\rangle=0$.
In this subsection we deal with the first step.

\begin{lemma}\label{lemma_compactreturn}
Let $\phi\in \mathcal{H}$ be additively compact (i.e. such that $V_A\phi=\phi$).
Then for any $p\in\D$
$$\pl_u\|A_u\phi-\phi\|=0$$
In other words, for all $\eps>0$ the set $S:=\{u\in K:\|A_u\phi-\phi\|<\eps\}$ is $DC^*$.
\end{lemma}

\begin{proof}
The orbit closure $\overline{\{A_u\phi:u\in R\}}$ of $\phi$ is trivially contained in the union $\bigcup_{u\in R}B(A_u\phi,\epsilon/2)$.
Hence, by compactness, there exists some finite set $F\subset R$ such that the union $\bigcup_{u\in F}B(A_u\phi,\epsilon/2)$ contains the whole orbit of $\phi$ under the additive sub-semigroup $S_A$.
Let $r:=|F|+1$.

Let $Z\subset K$ be an arbitrary subset with cardinality $|Z|=r$.
We claim that the set of finite sums $FS(Z)\cap S\neq\emptyset$.
Indeed, let $Z=\{z_1,...,z_r\}$, let $z_i'=z_1+...+z_i$ for each $i=1,...,r$ and note that $z_i-z_j\in FS(Z)$ for each $i>j$.
By the pigeonhole principle, there are $1\leq i<j\leq r$ such that $A_{z_i'}\phi$ and $A_{z_j'}\phi$ are in the same ball $B(A_u\phi,\epsilon/2)$ for some $u\in F$.
Thus $\|A_{z_i'}f-A_{z_j'}f\|<\epsilon$ and since the action of $S_A$ is an isometry of $\mathcal{H}$ we conclude that $\|A_{z_i'-z_j'}\phi-\phi\|<\epsilon$.
This implies that $z_i'-z_j'\in S$ and it proves the claim.

By Lemma \ref{lemma_centralIP0}, every $DC$ set contains $FS(Z)$ for some set $Z\subset R$ with $|Z|=r+1$. Therefore $S$ has nonempty intersection with every $DC$ set, and hence $S$ is $DC^*$ as desired.
\end{proof}

\begin{lemma}\label{lemma_large}
For all $p\in\D$ and $\phi,\psi\in \mathcal{H}$ we have
$$\pl_u\langle A_u(V_A\phi),M_u\psi\rangle=\langle V_A\phi,V_M\psi\rangle$$
\end{lemma}

\begin{proof}
We will assume, without loss of generality, that $\|\phi\|,\|\psi\|\leq1$.
In view of Lemma \ref{lemma_kronecker} we have
$$\pl_u\langle V_A\phi,M_u\psi\rangle=\Big\langle V_A\phi,\big(\pl_uM_u\psi\big)\Big\rangle=\langle V_A\phi,V_M\psi\rangle$$
Therefore, for every $\eps>0$, the set
$$S_1=\left\{u\in R:\left|\langle V_A\phi,M_u\psi\rangle-\langle V_A\phi,V_M\psi\rangle\right|<\frac\eps2\right\}$$
belongs to $p$.

Applying Lemma \ref{lemma_compactreturn} with $V_A\phi$ we get that the set $S_2:=\{u\in R   :\|A_uV_A\phi-V_A\phi\|<\eps/2\}$ is also in $p$.
Using the Cauchy-Schwarz inequality we have that for any $u\in S_2$
$$\left|\langle V_A\phi,M_u\psi\rangle-\langle A_uV_A\phi,M_u\psi\rangle\right|<\frac\eps2$$
Finally let $S:=S_1\cap S_2\in p$ and let $u\in S$.
We conclude that
$$\left|\langle A_uV_A\phi,M_u\psi\rangle-\langle V_A\phi,V_M\psi\rangle\right|<\eps$$
which finishes the proof.
\end{proof}

\subsection{Dealing with $\phi^\perp$ when $R$ is a field}
We now turn our attention to the weak mixing component $\phi^\perp$.
Dealing with this component 
in the general case requires some technical steps which obscure the main ideas.
In order to clarify these ideas we restrict our attention in this subsection to the case where $R$ is a field; the general case is treated in the next subsection.
(Of course the results of this subsection also follow logically from the results in the next one.)

We will use the following version of the van der Corput trick.
\begin{proposition}[cf. {\cite[Theorem 2.3]{Bergelson_McCutcheon07}}]\label{lema_pvdc}
Let $p\in\D$, let $\mathcal{H}$ be a Hilbert space, let $(a_u)_{u\in\Rs}$ be a bounded sequence in $\mathcal{H}$ indexed by $\Rs$.
If $\displaystyle\pl_u\langle a_{bu},a_u\rangle=0$ for all $b$ in a co-finite subset of $\Rs$ then
$\displaystyle\pl_ua_u=0$ in the weak topology of $\mathcal{H}$.
\end{proposition}

\begin{lemma}\label{cor_wm}
Let $K$ be a field, let $\mathcal{H}$ be a Hilbert space, let $(U_g)_{g\in\AK}$ be a unitary anti-representation of $\AK$ on $\mathcal{H}$ and let $\phi^\perp,\psi\in \mathcal{H}$, where we assume that $V_A\phi^\perp=0$.
Then, for all $p\in\D$ we have
$$\pl_u\langle A_u\phi^\perp,M_u\psi\rangle=0$$
\end{lemma}
\begin{proof}
Observe that, since we deal with an anti-representation, the distributive law (see \eqref{eq_distr}) takes the form
\begin{equation}\label{eq_distributivityinverted}
  A_vM_u=M_uA_{vu}
\end{equation}
for any $v\in K$ and $u\in K^*$.
Let $a_u=M_{1/u}A_u\phi^\perp$.
Then for all $b\in K\setminus\{-1,0,1\}$, using \eqref{eq_distributivityinverted} and the fact that isometries preserve scalar products we have
$$\langle a_{ub},a_u\rangle=\langle M_{1/ub}A_{ub}\phi^\perp,M_{1/u}A_u\phi^\perp\rangle=\langle A_{u(b-1/b)}\phi^\perp,M_b\phi^\perp\rangle$$
Therefore, it follows from Lemma \ref{lemma_kronecker} that for every $p\in\D$ we have
$$\pl_u\langle a_{ub},a_u\rangle=\left\langle\pl_uA_{u(b-1/b)}\phi^\perp,M_b\phi^\perp\right\rangle=\langle V_A\phi^\perp,M_b\phi^\perp\rangle=0$$
By Proposition \ref{lema_pvdc} we conclude that $\displaystyle\pl_uM_{1/u}A_u\phi^\perp=\pl_ua_u=0$.
Hence we have
\begin{eqnarray*}
  \pl_u\langle A_u\phi^\perp,M_u\psi\rangle&=&\pl_u\langle M_{1/u}A_u\phi^\perp,\psi\rangle\\&=&\left\langle\pl_uM_{1/u}A_u\phi^\perp,\psi\right\rangle=0
\end{eqnarray*}
\end{proof}

\subsection{Dealing with $\phi^\perp$ when $R$ is a general \FIID}

In this subsection we extend the scope of Lemma \ref{cor_wm} from the previous sub-section to the case when we have a general \FIID{} (not necessarily a field). Namely, we will prove:

\begin{lemma}\label{lemma_wm}
Assume $R$ is an \FIID, let $\mathcal{H}$ be a Hilbert space, let $(U_g)_{g\in\AR}$ be an isometric anti-representation of $\AR$ on $\mathcal{H}$ and let $\phi^\perp,\psi\in \mathcal{H}$.
Assume that $V_A\phi^\perp=0$.
Then, for all $p\in\D$ we have
$$\pl_u\langle A_u\phi^\perp,M_u\psi\rangle=0$$
\end{lemma}

In the proof of this lemma we will need a few facts about isometric anti-representations of $\AR$.
First observe that, unlike the case when $R$ is a field, $M_u$ is not necessarily invertible.
Thus its adjoint $M_u^T$ (defined so that $\langle M_u\phi,\psi\rangle=\langle\phi,M_u^T\psi\rangle$ for all $\phi,\psi\in \mathcal{H}$) may not be in $\AR$.
However, since $A_u$ is invertible (and hence unitary) we have the following distributivity relation:
\begin{lemma}\label{lemma_integers1}
  Under the assumptions of Lemma \ref{lemma_wm} we have
  $$A_{uv}M_u^T=M_u^TA_v$$
\end{lemma}
\begin{proof}
  We have, for any $\phi,\psi\in \mathcal{H}$
  \begin{eqnarray*}
    \langle A_{uv}M_u^T\phi,\psi\rangle=\langle\phi,M_uA_{-uv}\psi\rangle=\langle\phi,A_{-v}M_u\psi\rangle=\langle M_u^TA_v\phi,\psi\rangle.
  \end{eqnarray*}
  This implies the identity in question.
\end{proof}
Another difficulty which is present in our current context is the fact that the composition $M_nM_n^T$ is not necessarily the identity map.
The following lemma allows to circumvent this difficulty when $R$ is an \FIID{}.
\begin{lemma}\label{lemma_integers2}
  Under the assumptions of Lemma \ref{lemma_wm}, there exists an orthogonal projection $P:\mathcal{H}\to \mathcal{H}$ such that for every $\phi\in \mathcal{H}$ we have
  $$\pl_u\|M_uM_u^T\phi-P\phi\|=0$$
\end{lemma}
\begin{proof}
  Let $P_u=M_uM_u^T$.
  Since $M_u$ is an isometry, $P_u$ is the orthogonal projection onto the image of $M_u$.
  Observe that, in particular, the image of $P_{u_1u_2}$ is contained in the image of each $P_{u_i}$, $i=1,2$.

  Let $\{r_1,r_2,\dots\}$ be an arbitrary enumeration of the elements of $\Rs$ and let $u_n=\prod_{i=1}^nr_i$.
  Let $S_n$ be the image of $M_{u_n}$, so that $P_{u_n}$ is the orthogonal projection onto $S_n$.
  Note that $S_{n+1}\subset S_n$.
  Let $S=\bigcap_{n\geq1} S_n$ and let $P:\mathcal{H}\to S$ be the orthogonal projection.
  Let $E_0$ be an orthonormal basis for $S$ and, for each $n\geq1$ let $E_n$ be an orthonormal basis for $S_n\cap(S_{n+1})^\perp$.
  Thus $E=\bigcup_{n\geq0} E_n$ is an orthonormal basis for $\mathcal{H}$.
  Write $\phi$ in terms of the basis $E$ as $\phi=\sum_{n\geq0}\sum_{e\in E_n}c_ee$.
  For a fixed $\eps>0$ let $m\in\N$ be such that $\sum_{n\geq m}\sum_{e\in E_n}|c_e|^2<\epsilon^2$.

  Next, let $u$ be in the ideal $u_mR$.
  We have that the image of $P_u$ is contained in the image of $P_{u_m}$, so $P_uh\in S_m$ and hence
  $$P_u\phi=\sum_{e\in E_0}c_ee+\sum_{n=m}^\infty\sum_{e\in E_n}c_ee=P\phi+\sum_{n=m}^\infty\sum_{e\in E_n}c_ee$$
  Therefore $\|P_u\phi-P\phi\|<\eps$.
  Since the ideal $u_mR$ has finite index as an additive group, it follows from Lemma \ref{lemma_finiteindexip*} that it belongs to $p$.
  We conclude that $\pl M_nM_n^T\phi=\pl P_n\phi=P\phi$ in the strong topology, as desired.
\end{proof}

Finally, we need a strengthening of Lemma \ref{lemma_kronecker}.
\begin{definition}
  Let $R$ be an integral domain, let $b\in R$ and let $p\in\beta R$.
  Assume that $bR\in p$.
  Given a sequence $(x_u)_{u\in R}$ in a compact space $X$ we define $\pl_u x_{u/b}$ to be the point $x\in X$ such that for every neighborhood $U$ of $x$, the set $\{u\in bR:x_{u/b}\in U\}\in p$.
\end{definition}
\begin{lemma}\label{lemma_kroneckerideal}
  Let $R$ be an \FIID, let $p\in\D$ and let $k,b\in\Rs$.
  For any unitary anti-representation $(U_g)_{g\in\AR}$ of the semigroup $\AR$ on a Hilbert space $\mathcal{H}$ and any $\phi\in \mathcal{H}$ we have
  $$\pl_uA_{ku/b}\phi=V_A\phi\qquad\text{in the weak topology}$$
\end{lemma}
\begin{proof}
  First observe that the $\pl$ is well defined since the ideal $bR$ has finite index in $R$, $p$ belongs to the closure $\overline{\MI}$ of the additive minimal idempotents and hence, in view of Lemma \ref{lemma_finiteindexip*}, $bR\in p$.

  It follows from Lemma \ref{lemma_kplim} that $\pl_uA_{ku/b}\phi=k\pl_u A_{u/b}\phi$.
  Since, in view of Lemma \ref{lemma_upinmi}, $kp\in\overline{\MI}$, we can and will assume that $k=1$.
  Next, let $q=b^{-1}p$ be the ultrafilter defined so that $E\in q\iff bE\in p$.
  It follows from Lemma \ref{lemma_upinmi} that $q\in\overline{\MI}$.
  Therefore, it follows from Lemma \ref{lemma_kronecker} that for any $\psi\in \mathcal{H}$ and $\eps>0$ the set
  $$E=\{u\in R:\big|\langle A_u\phi-V_A\phi,\psi\rangle\big|<\eps\}\in q$$
  We conclude that
  $$bE=\{u\in bR:\big|\langle A_{u/b}\phi-V_A\phi,\psi\rangle\big|<\eps\}\in p$$
\end{proof}
We can now give a proof of Lemma \ref{lemma_wm}:
\begin{proof}[Proof of Lemma \ref{lemma_wm}]
Let $M_u^T$ denote the adjoint of $M_u$.
Then $\langle A_u\phi^\perp,M_u\psi\rangle=\langle M_u^TA_u\phi^\perp,\psi\rangle$, so the lemma will follow if we show that $\pl M_u^TA_u\phi^\perp=0$ (in the weak topology).
To do this we will use the van der Corput trick (Proposition \ref{lema_pvdc}), and so it suffices to show that
\begin{equation}\label{eq_integers14}
\pl_u\langle M_{ub}^TA_{ub}\phi^\perp,M_u^TA_u\phi^\perp\rangle=0\qquad\forall b\in R\setminus\{-1,0,1\}
\end{equation}
Since the operator $A_u$ is unitary we can rewrite the inner product in (\ref{eq_integers14}) as $\langle M_{ub}^TA_{ub}\phi^\perp,M_u^TA_u\phi^\perp\rangle=\langle A_{-u}M_uM_{ub}^TA_{ub}\phi^\perp,\phi^\perp\rangle$.
By \eqref{eq_distributivityinverted} we have $A_{-u}M_u=M_uA_{-u^2}$ (recall this is an anti-representation).
Also, assuming that $u\in bR$ and evoking Lemma \ref{lemma_integers1} we conclude that
$$\langle M_{ub}^TA_{ub}\phi^\perp,M_u^TA_u\phi^\perp\rangle=\langle M_uM_{ub}^T A_{ub-u/b}\phi^\perp,\phi^\perp\rangle=\langle A_{ub-u/b}\phi^\perp,M_bM_nM_n^T\phi^\perp\rangle$$

By Lemma \ref{lemma_kroneckerideal} we have that $\pl A_{ub-u/b}\phi^\perp=V_A\phi^\perp=0$ in the weak topology.
By Lemma \ref{lemma_integers2} we have that $\pl_u M_bM_uM_u^T\phi^\perp$ exists in the strong topology.
Thus we conclude that $\pl \langle A_{ub-u/b)}\phi^\perp,M_bM_nM_n^T\phi^\perp\rangle=0$, which gives (\ref{eq_integers14}) and finishes the proof.
\end{proof}

\subsection{Proofs of the main results}
We have now gathered all the ingredients necessary for the proofs of the main Theorems of the paper.
We start by proving Theorem \ref{thm_main}:

\begin{proof}[Proof of Theorem \ref{thm_main}]
Let $\phi^\perp=\phi-V_A\phi$, so that $V_A\phi^\perp=0$.
Using Lemmas \ref{lemma_large} and \ref{lemma_wm} we deduce that
$$\pl\langle A_u\phi,M_u\psi\rangle=\pl\langle A_uV_A\phi,M_u\psi\rangle+\langle A_u\phi^\perp,M_u\psi\rangle=\langle V_A\phi,V_M\psi\rangle$$
\end{proof}
As a corollary we now obtain the following:
\begin{theorem}\label{thm_mcentral}
Let $R$ be an \FIID, let $p\in\D$, let $(\Omega,\mu)$ be a probability space, let $(T_g)_{g\in\AR}$ be a measure preserving action of $\AR$ on $\Omega$, let $B\subset \Omega$ be a measurable set and let $\eps>0$.
Then the set
$$R(B,\eps):=\Big\{u\in R:\mu\big(A_u^{-1}B\cap M_u^{-1}B\big)\geq\mu(B)^2-\eps\Big\}$$ is $DC^*$ and, in particular, affinely syndetic.
\end{theorem}
\begin{proof}
  Let $\mathcal{H}=L^2(\Omega,\mu)$ and, for each $g\in\AR$, define the operator $(U_g\phi)(x)=\phi(T_gx)$. Observe that $U_gU_h=U_{hg}$, so this induces an isometric anti-representation $(U_g)_{g\in\AR}$ of $\AR$ in $\mathcal{H}$.
  Let $B\subset\Omega$. Observe that
  $$1_{T_g^{-1}B}(x)=1\iff T_gx\in B\iff 1_B(T_gx)=1\iff U_g1_B(x)=1$$
  Therefore $\mu(A_u^{-1}B\cap M_u^{-1}B)=\int_\Omega A_u1_B\cdot M_u1_Bd\mu=\langle A_u1_B,M_u1_B\rangle$.
  It follows from Theorem \ref{thm_main} that for any $\epsilon>0$ the set
  $$\big\{u\in R:\langle A_u1_B,M_u1_B\rangle\geq\langle V_A1_B,V_M1_B\rangle-\epsilon\big\}$$
  is $DC^*$.
  Finally, it follows from Corollary \ref{lemma_positive} (applied with $\phi=1_B$ and $\psi\equiv1$) that
  $$\langle V_A1_B,V_M1_B\rangle\geq\mu(B)^2$$
\end{proof}
Observe that Theorem \ref{thm_Rintersection} easily follows from Theorem \ref{thm_mcentral}.
Indeed, given $p\in\D$ it follows from the definition of $DC^*$ sets and Theorem \ref{thm_mcentral} that $R(B_i,\delta)\in p$ for every $i$.
Therefore also the intersection $R=R(B_1,\delta)\cap\cdots\cap R(B_t,\delta)$ belongs to $p$.
Since $p\in\D$ was arbitrary, it follows that $R$ is itself a $DC^*$ set.
Finally, Remark \ref{rmrk_DC*syndetic} implies that $R$ must be affinely syndetic.

We now present the main combinatorial corollary  of Theorem \ref{thm_mcentral}:
\begin{theorem}\label{cor_integer}
  Let $K$ be a countable field and let $R\subset K$ be a sub-ring which is a \FIID.
  Let $E\subset K$ with $\bar d_{(F_N)}(E)>0$ for some double F\o lner sequence $(F_N)$ and let $\epsilon>0$.
  Then the set
  \begin{equation}\label{eq_corollaryintegerset}\Big\{u\in R:\bar d_{(F_N)}\big((E-u)\cap(E/u)\big)>\bar d_{(F_N)}(E)^2-\epsilon\Big\}\end{equation}
  is $DC^*$ and, in particular, affinely syndetic in $R$.
\end{theorem}
\begin{proof}
  Using the correspondence principle (Theorem 2.8 in \cite{Bergelson_Moreira15}) one can construct a measure preserving action $(T_g)_{g\in\AK}$ of $\AK$ on a probability space $(\Omega,{\mathcal B},\mu)$ and a set $B\in{\mathcal B}$ such that $\mu(B)=\bar d_{(F_N)}(E)$ and, for each $u\in K^*$
  $$\bar d_{(F_N)}\big((E-u)\cap(E/u)\big)\geq\mu(A_u^{-1}B\cap M_u^{-1}B)$$
  The result now follows from Theorem \ref{thm_mcentral}.
\end{proof}

One can deduce parts (2) and (3) of Theorem \ref{cor_introinteger} from Theorem \ref{cor_integer} using the fact that for any finite partition of a countable field, one of the cells of the partition has positive upper density with respect to a double F\o lner sequence.
Then using that cell $C_i$ of the partition as $E$, for any element $n$ of the (non-empty) set defined in \eqref{eq_corollaryintegerset} and for any $x$ in the (non-empty) intersection $(C_i-n)\cap(C_i/n)$ we have $\{x+n,xn\}\subset C_i$.

To deduce part (1) of Theorem \ref{cor_introinteger}, one needs an additional fact:
\begin{proposition}\label{prop_Nismultcentral*}
  The subset $\N$ of the ring $\Z$ belongs to every non-principal multiplicative idempotent.
\end{proposition}
\begin{proof}
  Let $p\in\beta\Z$ be a non-principal multiplicative idempotent.
  Assume, for the sake of a contradiction, that $\N\notin p$.
  Then $-\N\in p=pp$, which by definition implies that $\{n\in\Z^*:-\N/n\in p\}\in p$.
  Observe that
  $$-\N/n=\{a\in\Z^*:an\in-\N\}=\left\{\begin{array}{rl}\N&\text{ if }n\in-\N\\-\N&\text{ if }n\in\N \end{array}\right.$$
  Therefore $\{n\in\Z^*:-\N/n\in p\}=\N\notin p$, which is the desired contradiction.
\end{proof}

To deduce part (1) of Theorem \ref{cor_introinteger} one applies Theorem \ref{cor_integer} with $K=\Q$, $R=\Z$ and $E$ being a cell of the partition with positive upper density with respect to a double F\o lner sequence.
The set $S$ defined by \eqref{eq_corollaryintegerset} is $DC^*$ in $\Z$, which means that for any $p\in\D$ we have $S\in p$.
Since any $p\in\D$ is a non-principal multiplicative idempotent, it follows from Proposition \ref{prop_Nismultcentral*} that also $\N\in p$, and therefore $S\cap\N\in p$ and hence is non-empty.
For any $n$ in that intersection the set $(E-n)\cap(E/n)$ is non-empty and any $x$ in this intersection yields $\{x+n,xn\}\subset E$.

\section{Notions of largeness and configurations $\{xy,x+y\}$ in $\N$}\label{sec_misc}
In this section we discuss notions of largeness which guarantee the presence of configurations of the form $\{x+y,xy\}$.

It is a trivial observation that the set of odd numbers in $\N$ or in $\Z$ does not contain pairs $\{x+y,xy\}$.
Therefore, additively syndetic sets (i.e. sets which are syndetic with respect to the additive semigroup) do not contain, in general, configurations $\{x+y,xy\}$.
It is thus somewhat surprising that multiplicatively syndetic subsets in any integral domain do contain such patterns:
\begin{theorem}
  Let $R$ be an infinite countable integral domain and let $S\subset R^*$ be multiplicatively syndetic (i.e. syndetic as a subset of the semigroup $(R^*,\cdot)$).
  Then $S$ contains (many) pairs of the form $\{x+y,xy\}$.
\end{theorem}
\begin{proof}
  Let $F\subset R^*$ be a finite set such that $R^*=\bigcup_{n\in F} S/n$ (the existence of such $F$ is equivalent, by definition, to the statement that $S$ is multiplicatively syndetic).
  Thus $R^*$ is finitely partitioned into multiplicative shifts of $S$ and hence that there exist (many) $a,b\in R^*$ such that $a+bF\subset S$
  \footnote{This is a well known extension of van der Waerden's theorem in arithmetic progressions. One way to prove this is to apply the Hales-Jewett theorem, as in the proof of Proposition 4.4 in \cite{Bergelson_Johnson_Moreira15}, where a stronger statement is proved.}.
  Since $ab\in R^*=\bigcup_{n\in F} S/n$, there exist some $n\in F$ such that $abn\in S$.
  We conclude that $\big\{a+bn,a(bn)\big\}\subset S$ as desired.
\end{proof}

While it is not hard to see that there exist partitions of $\N$ or $\Z$ with none of the cells of the partition being multiplicatively syndetic, it is a classical fact that for any finite partition of a semigroup, one of the cells is piecewise syndetic\footnote{A subset $E$ of a commutative semigroup is called \emph{piecewise syndetic} if it is the intersection of a syndetic set and a thick set.}.
One could then hope that any multiplicatively piecewise syndetic subset of $R^*$ contains a pattern $\{x+y,xy\}$.
Unfortunately, the next example refute this assertion.

\begin{theorem}\label{thm_thickbad}
There exists a set $E\subset\N$ which is additively thick and multiplicatively thick (and so, in particular, $E$ is a multiplicatively piecewise syndetic subset of $\N$) but does not contain a pair $\{x+y,xy\}$ with $x,y>2$.
\end{theorem}

\begin{proof}
  Let $(p_N)$ be a sequence of primes such that $p_1=5$ and, for each $N\in\N$, we have $p_{N+1}>4(Np_N)^4$.
For each $N\in\N$, let
$$E_{2N-1}=p_N[1,N]\qquad\text{ and }\qquad E_{2N}=[(Np_N)^2+1,2(Np_N)^2-3]$$
where we use the notation $[a,b]$ to denote the set $\{a,a+1,\dots,b\}$.
Let $E=\bigcup E_N$.
It follows directly from the construction that $E$ is additively thick as a subset of either $\N$ or $\Z$ and is multiplicatively thick as a subset of $\N$.
Moreover, $E\cup(-E)$ is a multiplicatively thick subset of $\Z^*$.
Since $\N$ is a multiplicatively syndetic subset of $\Z^*$, it follows that $E$ is a multiplicative piecewise syndetic subset of $\Z^*$.

We first show that no set $E_{2N}$ contains a pair $\{x+y,xy\}$: assume that $a=x+y\in E_{2N}$ and $x,y\geq2$.
Let $b=xy$.
Then $b\geq2(a-2)\geq 2[(Np_N)^2+1-2]=2(Np_N)^2-2$, so $b$ is too large to be in $E_{2N}$.

Next we show that no set $E_{2N-1}$ contains such a pair.
Assume $xy\in E_{2N-1}$, say $xy=np_N$, then without loss of generality we have $x=p_Nd$ and $y=n/d$ for some divisor $d$ of $n$.
But then $x+y<p_N(d+1)$ because $n/d\leq N<p_N$.
Hence $x+y\notin E_{2N-1}$.

For each $N\in\N$ we have $(\max E_{2N-1})^2=(Np_N)^2<(Np_N)^2+1=\min E_{2N}$ and $(\max E_{2N})^2=(2(Np_N)^2-3)^2<4(Np_N)^4<p_{N+1}=\min E_{2N+1}$.
Fix a pair $x,y\in\N$ with both $x,y\geq2$, let $a=xy$ and $b=x+y$.
We observe that $b\leq a\leq(b/2)^2$.

If $b\in E$, say $b\in E_n$, then $\min E_n\leq b\leq a\leq (b/2)^2<[(\max E_n)/2]^2<\min E_{n+1}$ so $a$ can not be in $E_m$ for any $m\neq n$.
Since we already showed that $a\notin E_n$ (otherwise $E_n$ would contain $\{b,a\}=\{x+y,xy\}$), we conclude that $a\notin E$ and this finishes the proof.
\end{proof}

We observe that the complement $\tilde E=\N\setminus E$ of the set constructed in Theorem \ref{thm_thickbad} is also rather large.
In particular $\bar d(\tilde E)=1$, where, as usual, for a subset $S\subset\N$, $\bar d(S)$ denotes the upper density,
$$\bar d(S)=\limsup_{N\to\infty}\frac{\big|S\cap\{1,\dots,N\}\big|}N.$$
The next result shows that sets having upper density $1$ are large not only additively, but also multiplicatively.
\begin{theorem}\label{thm_naturalthick}
  Let $E\subset\N$ satisfy $\bar d(E)=1$.
  Then $E$ is affinely thick.
\end{theorem}
\begin{proof}
  Since $\bar d$ is the upper density with respect to an additive F\o lner sequence, it is not hard to see that $\bar d\big((E-n)\cap E)=1$ for any $n\in\N$.
  We claim that also $\bar d\big((E/n)\cap E)=1$ for any $n\in\N$.

  Assuming the claim for now, let $F=\{g_1,\dots,g_k\}\subset\AN$ be an arbitrary finite set.
  We can write each $g_i$ as the map $g_i:x\mapsto a_ix+b_i$.
  Let $E_0=E$ and, for each $i=1,\dots,k$, let $A_i=\big((E_{i-1}-b_i)\cap E_{i-1}\big)$ and $E_i=\big((A_i/a_i)\cap A_i)$.
  It follows by induction that each of the sets $E_i$, $A_i$ satisfies $\bar d(E_i)=\bar d(A_i)=1$.
  Take $x\in E_k$, we will show that $g_i(x)\in E$ for every $i$.
  Indeed, $x\in E_k\subset E_i=\big((A_i/a_i)\cap A_i)$, so $a_ix\in A_i=\big((E_{i-1}-b_i)\cap E_{i-1}\big)$ and hence $a_ix+b_i=g_i(x)\in E_{i-1}\subset E$ as desired.

  Now we prove the claim.
  We will write $[1,x]$ to denote the set $\{1,2,\dots,\lfloor x\rfloor\}$, where $\lfloor x\rfloor$ is the largest integer no bigger than $x$.

  Let $n\in\N$ and take $\epsilon>0$ arbitrary.
  For some arbitrarily large $N\in\N$ we have
  $$|E\cap[1,N]|>\left(1-\frac\epsilon{2n}\right)N=N-\frac{\epsilon N}{2n}$$
  This implies that
  $$|nE\cap[1,N]|=\bigg|E\cap\left[1,\tfrac Nn\right]\bigg|>\frac Nn-\frac{\epsilon N}{2n}$$
  Using the general fact that $|X\cup Y|+|X\cap Y|=|X|+|Y|$ we deduce that $nE\cap E\cap[1,N]=(nE\cap[1,N])\cap(E\cap[1,N])$ has cardinality
  \begin{eqnarray*}
    \displaystyle \big|nE\cap E\cap[1,N]\big|&=&\displaystyle \big|E\cap[1,N]\big|+\big|nE\cap[1,N]\big|-\Big|\big(nE\cap[1,N]\big)\cup\big(E\cap[1,N]\big)\Big|\\&\geq& \displaystyle \quad N-\frac{\epsilon N}{2n}\ \ +\ \ \frac Nn-\frac{\epsilon N}{2n}\ \ -\qquad N\\&=&\displaystyle \frac Nn(1-\epsilon)
  \end{eqnarray*}
  Dividing by $n$ (and observing that every number in the intersection $nE\cap E\cap[1,N]$ is divisible by $n$) we deduce that
  $$|E\cap(E/n)\cap[1,N/n]|=|nE\cap E\cap[1,N]|\geq\frac Nn(1-\epsilon)$$
  As $N$ can be taken arbitrarily large and $\epsilon$ arbitrarily small we conclude that $\bar d\big(E\cap(E/n)\big)=1$, proving the claim.
\end{proof}

It is clear that, for any $y\in\N$, any affinely thick set contains configurations of the form $\{x+y,xy\}$.
 This observation applies, in particular, to the complement $\tilde E$ of the set $E$ constructed in Theorem \ref{thm_thickbad}.

Recall now the notion of $DC$ set (see Definition \ref{def_central}) and observe that for any finite partition of $\N$ one of the cells is a $DC$ set.
It follows from (the proof of) \cite[Corollary 5.5]{Bergelson_Hindman90} that any $DC$ set is both additively piecewise syndetic and multiplicatively piecewise syndetic.
For a partition of $\N$ into two cells, one has the following dichotomy: either one of the cells has upper density $1$ (in which case Theorem \ref{thm_naturalthick} assures us that it contains configurations $\{x+y,xy\}$) or both cells have positive lower density.
In view of this observation we make the following conjecture:
\begin{conjecture}\label{conj_thick}
  Let $E\subset\N$ be additively and multiplicatively piecewise syndetic and have positive lower density.
  Then $E$ contains many configurations of the form $\{x+y,xy\}$.
\end{conjecture}
While Conjecture \ref{conj_thick} implies that for any partition of $\N$ into two cells, one of the cells contains many configurations $\{x+y,xy\}$, the property of having positive lower density is not stable under partitions.
Indeed it is not hard to construct a partition of $\N$ into two sets, both with $0$ lower density.
However, for any finite partition of a $DC$ set, one of the cells is still a $DC$ set.
Observe that the example $E$ constructed in the proof of the Theorem \ref{thm_thickbad} can be split into two sets $E=E_A\cup E_M$ such that $E_A$ is additively thick, but has density $0$ with respect to any multiplicative F\o lner sequence, and $E_M$ is multiplicatively thick but has density $0$ with respect to any additive F\o lner sequence.
Therefore $E$ is very far from being a $DC$ set.
This observation leads to the following conjecture:
\begin{conjecture}\label{conj_dc}
  Every $DC$ set in $\N$ contains a configuration $\{x+y,xy\}$.
\end{conjecture}
Observe that Conjecture \ref{conj_dc} implies that for any finite partition of $\N$, one of the cells contains plenty of configurations $\{x+y,xy\}$.

\bibliography{refs-joel}
\bibliographystyle{plain}
\end{document}